\documentclass{amsart}

\newtheorem{theorem}{Theorem}
\newtheorem{lemma}[theorem]{Lemma}
\newtheorem{proposition}[theorem]{Proposition}
\newtheorem{corollary}[theorem]{Corollary}

\theoremstyle{definition}

\newtheorem{example}[theorem]{Example}

\usepackage{amscd,amssymb}

\begin{document}

\title[Invariant theory of bicommutative algebras]
{Invariant theory of free bicommutative algebras}

\author[Vesselin Drensky]
{Vesselin Drensky}
\address{Institute of Mathematics and Informatics,
Bulgarian Academy of Sciences,
1113 Sofia, Bulgaria}
\email{drensky@math.bas.bg}

\subjclass[2010]{17A36; 17A30; 17A50; 13A50; 15A72.}
\keywords{Free bicommutative algebras, noncommutative invariant theory.}

\maketitle

\centerline{\it Dedicated to Alberto Elduque on the occasion of his 60th birthday}

\begin{abstract}
The variety $\mathfrak B$ of bicommutative algebras consists of all nonassociative algebras satisfying the polynomial identities
of right- and left-commutativity $(x_1x_2)x_3=(x_1x_3)x_2$ and $x_1(x_2x_3)=x_2(x_1x_3)$.
Let $F_d({\mathfrak B})$ be the free $d$-generated bicommutative algebra over a field $K$ of characteristic 0.
We study the algebra $F_d({\mathfrak B})^G$ of invariants of a subgroup $G$ of the general linear group $GL_d(K)$.
When $G$ is finite we search for analogies of classical results of invariant theory of finite groups acting on polynomial algebras:
the Endlichkeitssatz of Emmy Noether, the Molien formula and the Chevalley-Shephard-Todd theorem
and show the similarities and the differences in the case of bicommutative algebras.
We also describe the symmetric polynomials in $F_d({\mathfrak B})$.
\end{abstract}

\section{Introduction}

A nonassociative algebra $R$ over a field $K$ is called right-commutative if it satisfies the polynomial identity $(x_1x_2)x_3=(x_1x_3)x_2$.
Similarly we define left-commutative algebras. Algebras which are both right- and left-commutative are called bicommutative.
The algebra $R$ is right-symmetric if it satisfies the polynomial identity
$(x_1,x_2,x_3)=(x_1x_3,x_2)$, where $(x_1,x_2,x_3)=(x_1x_2)x_3-x_1(x_2x_3)$ is the associator.
For the first time oneside-commutative and oneside-symmetric algebras appear in the paper by Cayley \cite{Ca} in 1857.
Left-commutative algebras which are right-symmetric are Novikov (or Gelfand-Dorfman-Novikov) algebras.
These algebras or their opposite were introduced independently by Gelfand and Dorfman \cite{GD}
in their algebraic approach to the Hamiltonian operator in mechanics
and by Balinskii and Novikov \cite{BN} in relation with hydrodynamics.
The study of bicommutative algebras started by Dzhumadil'daev and Tulenbaev \cite{DT} in 2003,
see also Dzhumadil'daev Ismailov and Tulenbaev \cite{DIT}
and under the name LR-algebras by Burde, Dekimpe and Deschamps \cite{BDD} in 2009.
Both Novikov and bicommutative algebras have a rich theory
which demonstrates the powerful methods of combinatorial nonassociative ring theory.
Additionally, there are nice history and beautiful combinatorics
including the realization of free Novikov and bicommutative algebras in terms of rooted trees,
see the introductions of the papers
by Dzhumadil'daev and L\"ofwall \cite{DL} and the recent preprint by Bokut, Chen and Zhang \cite{BCZ}.

Dzhumadil'daev and Tulenbaev \cite{DT} proved that the square $F_d^2({\mathfrak B})$ is a commutative and associative algebra.
Hence it is naturally to expect that the free bicommutative algebra $F_d({\mathfrak B})$ has many properties typical for polynomial algebras.
The idea to apply methods of commutative algebra in the study of polynomial identities of bicommutative algebras
was developed by Drensky and Zhakhayev \cite{DZ} and further in \cite{D3}.
Recently Shestakov and Zhang \cite{SZ} studies automorphisms of free bicommutative algebras in the spirit of automorphisms of polynomial algebras.
Bai, Chen and Zhang \cite{BaCZ} showed that a finitely generated bicommutative algebra has a finite Gr\"obner-Shirshov basis
and its Gelfand-Kirillov dimension is a nonnegative integer.

In the sequel we shall assume that $K$ is a field of characteristic 0.
In classical invariant theory the general linear group $GL_d(K)$ acts canonically on the $d$-dimensional vector space
$V_d$ with basis $\{v_1,\ldots,v_d\}$ and on the polynomial algebra $K[X_d]=K[x_1,\ldots,x_d]$ in $d$ variables by
\[
g(f(v))=f(g^{-1}(v)),\quad g\in GL_d(K), v\in V_d,
\]
and the functions $x_i:V_d\to K$ are defined by
\[
x_i(v)=\xi_i,\quad i=1,\ldots,d,\text{ where }v=\sum_{j=1}^d\xi_jv_j,\xi_j\in K, j=1,\dots,d.
\]
In one of the main branches of noncommutative invariant theory one replaces the polynomial algebra $K[X_d]$ with the free algebra $F_d({\mathfrak V})$
of a variety of algebras $\mathfrak V$.
For our purposes it is more convenient to assume that $GL_d(K)$ acts canonically on the vector space $KX_d$ with basis $X_d=\{x_1,\ldots,x_d\}$
and this action is extended diagonally on the algebra $K[X_d]$ by
\[
g(f(X_d))=g(f(x_1,\ldots,x_d))=f(g(x_1),\ldots,g(x_d)),\quad g\in GL_d(K), f(X_d)\in K[X_d].
\]
For a subgroup $G$ of $GL_d(K)$ the algebra of $G$-invariants is
\[
K[X_d]^G=\{f(X_d)\in K[X_d]\mid g(f)=f\text{ for all } g\in G\}.
\]
Similarly, we assume that the $d$-generated free algebra $F_d({\mathfrak V})$ in a variety $\mathfrak V$ is generated by $X_d$,
with the diagonal action of $GL_d(K)$ on $F_d({\mathfrak V})$ and the algebra of $G$-invariants of a subgroup $G$ of $GL_d(K)$ is
\[
F_d({\mathfrak V})^G=\{f\in F_d({\mathfrak V})\mid g(f)=f\text{ for all } g\in G\}.
\]
Since $\text{char}(K)=0$, $F_d({\mathfrak V})$ is a $\mathbb Z$-graded vector space and the same holds for $F_d({\mathfrak V})^G$.
The Hilbert (or Poincar\'e) series of $F_d^G({\mathfrak V})$ is the formal power series
\[
H(F_d({\mathfrak V})^G,t)=\sum_{n\geq 0}\dim(F_d^{(n)}({\mathfrak V})^G)t^n,
\]
where $F_d^{(n)}({\mathfrak V})^G$ is the homogeneous component of degree $n$ in $F_d({\mathfrak V})^G$.

The following problems are among the main ones in commutative invariant theory:
\begin{itemize}
\item Find classes of groups $G$ for which the algebra of invariants $K[X_d]^G$ is finitely generated.
Find explicitly subgroups $G$ of $GL_d(K)$ with the property that $K[X_d]^G$ is not finitely generated.
\item When $K[X_d]^G$ is finitely generated find its generators and the defining relations between them.
\item Find conditions which guarantee that $K[X_d]^G$ is isomorphic to a polynomial algebra,
i.e. it has a system of algebraically independent generators.
\item Prove the rationality of the Hilbert series of $K[X_d]^G$. Calculate the explicit form of the Hilbert series.
\end{itemize}

These problems are also in the center of noncommutative invariant theory.
But an essential difference is that the properties of $F_d({\mathfrak V})^G$ depend on two parameters --
the group $G$ and the variety $\mathfrak V$. For an idea for the spirit of the investigations when $K[X_d]$
is replaced by the free associative algebra $K\langle X_d\rangle$
or the relatively free algebra $F_d({\mathfrak V})$ of a variety $\mathfrak V$ of associative algebras see the surveys
by Formanek \cite{F} and the author \cite{D1}. See also the papers by Benanti, Boumova, Drensky, Genov and Koev \cite{BBDGK}
and by Drensky and Hristova \cite{DH} for the computing of the Hilbert series $F_d({\mathfrak V})^G$ when $G$ is one of the classical groups.

In \cite{D2} the author studied invariants of groups $G$ acting on free algebras $F_d({\mathfrak V})$ of varieties of right-symmetric and Novikov algebras.
It has turned out that $F_d({\mathfrak V})^G$ is not finitely generated for a large class of groups which includes all finite groups
if the variety $\mathfrak V$ contains all left-nilpotent of class 3 (i.e. satisfying the identity $x_1(x_2x_3)=0$)
right-symmetric or Novikov algebras.

In the present paper we study the algebra of invariants $F_d({\mathfrak B})^G$ of a finite group $G$ acting on the free bicommutative algebra $F_d({\mathfrak B})$.
Our main results are for  $G\not=\langle 1\rangle$ and are the following:
\begin{itemize}
\item The algebra $F_d({\mathfrak B})^G$ is not finitely generated but $F_d^2({\mathfrak B})^G$ is a finitely generated $K[Y_d,Z_d]^G$-module
under a natural action of $K[Y_d,Z_d]$ on $F_d^2({\mathfrak B})$.
\item We establish an analogue of the Molien formula for the Hilbert series $H(F_d({\mathfrak B})^G,t)$.
\item There is no analogue of the Chevalley-Shephard-Todd theorem and the algebra $F_d({\mathfrak B})^G$ is never free in $\mathfrak B$.
\end{itemize}
As an application of the obtained results for $F_d({\mathfrak B})^G$ we describe the symmetric polynomials in $F_d({\mathfrak B})$.

There are other classes of groups $G$ for which the algebra $K[X_d]^G$ is finitely generated. Such are
the linearly reductive linear algebraic groups acting rationally on $KX_d$ and their maximal unipotent subgroups.
Hence for all such groups some of the results above and especially Theorem \ref{analogue of Endlichkeitssatz} (ii) also hold.

\section{Preliminaries}

In what follows $K$ is a field of characteristic 0 and $F_d({\mathfrak B})$ is the free bicommutative algebra of finite rank $d$
freely generated by $X_d=\{x_1,\ldots,x_d\}$.
Thee fact that $F_d^2({\mathfrak B})$ is a commutative and associative algebra established in \cite{DT} was used in \cite{DZ} for the following realization of $F_d({\mathfrak B})$.

Let $K[Y_d,Z_d]=K[y_1,\ldots,y_d,z_1,\ldots,z_d]$ be the polynomial algebra in $2d$ commutative and associative variables
and let $\omega(K[Y_d])$ and $\omega(K[Z_d])$ be the augmentation ideals (consisting of the polynomials without constant terms) of $K[Y_d]$ and $K[Z_d]$, respectively.
As a vector space the algebra $R_d$ is a direct sum of $\omega(K[Y_d])\omega(K[Z_d])\subset K[Y_d,Z_d]$ and $KX_d$ and has a basis consisting of the monomials in $X_d$ and
\[
Y_d^{\alpha}Z_d^{\beta}=y_1^{\alpha_1}\cdots y_d^{\alpha_d}z_1^{\beta_1}\cdots z_d^{\beta_d},\quad
\vert\alpha\vert=\sum_{i=1}^d\alpha_i>0,\vert\beta\vert=\sum_{j=1}^d\beta_j>0.
\]
The multiplication in $R_d$ is defined by:
\[
\begin{array}{c}
x_ix_j=y_iz_j,\\
\\
x_i(Y_d^{\alpha}Z_d^{\beta})=y_iY_d^{\alpha}Z_d^{\beta},\\
\\
(Y_d^{\alpha}Z_d^{\beta})x_j=Y_d^{\alpha}Z_d^{\beta}z_j,\\
\\
(Y_d^{\alpha}Z_d^{\beta})(Y_d^{\gamma}Z_d^{\delta})=Y_d^{\alpha+\gamma}Z_d^{\beta+\delta}.\\
\end{array}
\]
Clearly, the algebra $R_d$ is generated by the set $X_d$, the multiplication in $R_d^2$ is as in the polynomial algebra $K[Y_d,Z_d]$
and $R_d^2$ has the structure of a $K[X_d]$-bimodule.

\begin{proposition} {\rm (\cite{DZ})}
The algebra $F_d({\mathfrak B})$ is isomorphic to the algebra $R_d$ both as an algebra and as a multigraded vector space
with an isomorphism defined by $x_i\to x_i$, $i=1,\ldots,d$.
\end{proposition}

Till the end of the paper we shall work in $R_d$ instead of in $F_d({\mathfrak B})$ but shall state the results for $F_d({\mathfrak B})$.
We assume that the group $GL_d(K)$ acts on $KY_d$ and $KZ_d$ in the same way as on $KX_d$.

\section{The finite generation problem}

The famous Endlichkeitssatz of Emmy Noether \cite{No} states the following.

\begin{theorem}
For any finite subgroup $G$ of $GL_d(K)$ the algebra of invariants $K[X_d]^G$ is finitely generated.
\end{theorem}

In the special case of finite groups the following result for $F_d({\mathfrak B})^G$ may be considered as an analogue of the theorem of Emmy Noether.

\begin{theorem}\label{analogue of Endlichkeitssatz}
Let $G$ be a subgroup of $GL_d(K)$.
\begin{itemize}
\item[(i)] If every system of generators of the algebra of invariants $K[X_d]^G$ contains polynomials which are of degree $\geq 2$,
then the algebra of invariants $F_d({\mathfrak B})^G$ is not finitely generated.
\item[(ii)] If the algebra $K[Y_d,Z_d]^G$ is finitely generated, then $F_d^2({\mathfrak B})^G$ is a finitely generated $K[Y_d,Z_d]^G$-module.
\end{itemize}
\end{theorem}

\begin{proof}
(i) Changing the basis of the vector space $KX_d$ we may assume that
the vector space $(KX_d)^G$ of linear $G$-invariants has a basis $X_m=\{x_1,\ldots,x_m\}$ and $m<d$
because by assumption $K[X_d]^G$ contains polynomials which do not belong to $K[X_m]\subset K[X_d]^G$.
Therefore $K[X_d]^G$ is generated by $x_1,\ldots,x_m$ and by homogeneous polynomials $h(X_d)$
which depend essentially on some of the variables $x_{m+1},\ldots,x_d$.
Let $h_0(X_d)\in K[X_d]^G$ be such a polynomial and let for example it depends also on $x_d$.
As we already stated, we shall work in the algebra $R_d$ instead of in $F_d({\mathfrak B})$. Let us assume that $R_d^G$ is finitely generated.
Hence it has a system of generators consisting of $x_1,\ldots,x_m$ and $f_1(Y_d,Z_d),\ldots,f_n(Y_d,Z_d)\in (R_d^2)^G$.
We may assume that all $f_i(Y_d,Z_d)$ are homogeneous with respect to $Y_d$ and to $Z_d$.
Since $h_0(Y_d)h_0(Z_d)$ belongs to $R_d^G$ and depends on the variables $y_d$ and $z_d$,
it does not belong to the subalgebra of $R_d$ generated by $X_m$ and $R_d^G$ has at least one generator $f_i(Y_d,Z_d)\in (R_d^2)^G$. Let
\[
k_{\max}=\max\{\deg_{Z_d}(f_i(Y_d,Z_d))\mid i=1,\ldots,n\}.
\]
Since $\deg_{Y_d}(f(Y_d,Z_d))\geq 1$ for all nonzero $f(Y_d,Z_d)\in R_d^2$
and the same holds for $\deg_{Z_d}(f_i(Y_d,Z_d))$
\[
k_{\min}=\min\{\deg_{Y_d}(f(Y_d,Z_d))\mid 0\not=f(Y_d,Z_d)\in (R_d^2)^G\}\geq 1
\]
and $\deg_{Y_d}(f_0(Y_d,Z_d))=k$ for some $f_0(Y_d,Z_d)\in (R_d^2)^G$.
Then all polynomials
\[
v_p(Y_d,Z_d)=f_0(Y_d,Z_d)h_0^p(Z_d),\quad p=0,1,2,\ldots,
\]
are $G$-invariants with the property $\deg_{Y_d}(v_p(Y_d,Z_d))=k_{\min}$. Hence they can be expressed as linear combinations of
\[
u_{(\alpha,\beta,\gamma)}(Y_d,Z_d)=\prod_{i=1}^my_i^{\alpha_i}\prod_{j=1}^nf_j^{\gamma_j}(Y_d,Z_d)\prod_{i=1}^mz_i^{\beta_i}
\]
such that
\[
\deg_{Y_d}(u_{(\alpha,\beta,\gamma)})=\sum_{i=1}^m\alpha_i+\sum_{j=1}^n\gamma_j\deg_{Y_d}(f_j(Y_d,Z_d))=k_{\min}.
\]
All $u_{(\alpha,\beta,\gamma)}(Y_d,Z_d)$ are homogeneous both in $Y_d$ and $Z_d$. If $\gamma_1+\cdots+\gamma_n>1$, then
\[
\deg_{Y_d}(u_{(\alpha,\beta,\gamma)}(Y_d,Z_d))\geq 2k_{\min}>k_{\min}
\]
which is impossible. Hence either $\gamma_1+\cdots+\gamma_n=1$ and $\alpha_1=\cdots=\alpha_m=0$
or $\gamma_1=\cdots=\gamma_n=0$ and $\alpha_1+\cdots+\alpha_m=k_{\min}$. In both cases
\[
\deg_{z_d}(u_{(\alpha,\beta,\gamma)}(Y_d,Z_d))\leq \max\{\deg_{Z_d}(f_i(Y_d,Z_d))\mid i=1,\ldots,n\}=k_{\max}.
\]
Since $h_0(X_d)$ essentially depends on $x_d$ we obtain that
$\deg_{z_d}(v_p(Y_d,Z_d))\geq p$ which is impossible for $p>k_{\max}$ and this completes the proof of (i).

(ii) Let the algebra $K[Y_d,Z_d]^G$ be finitely generated. The multiplication in $K[Y_d,Z_d]$ defines a natural action on $F_d^2({\mathfrak B})$:
\[
y_i\ast f(X_d)=x_if(X_d),z_i\ast f(X_d)=f(X_d)x_i,\quad f(X_d)\in F_d^2({\mathfrak B}),i=1,\ldots,d.
\]
Translating in terms of $R_d$ it is sufficient to show that $(R_d^2)^G$ is a finitely generated $K[Y_d,Z_d]^G$-module.
Since $R_d^2$ is an ideal of $K[Y_d,Z_d]$, we obtain that $(R_d^2)^G$ is an ideal of the finitely generated algebra $K[Y_d,Z_d]^G$.
By the Basissatz \cite{H} of Hilbert $(R_d^2)^G$ is a finitely generated ideal, i.e. a finitely generated $K[Y_d,Z_d]^G$-module.
\end{proof}

In order to specify Theorem \ref{analogue of Endlichkeitssatz} in the case of finite groups we need the following easy lemma.

\begin{lemma}\label{integral polynomials}
Let $G$ be a finite subgroup of $GL_d(K)$. Then the algebra $K[Y_d,Z_d]$ is integral over the subalgebra $K[Y_d]^GK[Z_d]^G$.
\end{lemma}

\begin{proof}
Let $G=\{g_1,\ldots,g_n\}$. Then $y_i$ satisfies the equation
\[
\prod_{j=1}^n(x-g_j(y_i))=\sum_{k=0}^n(-1)^ke_k(g_1(y_i),\ldots,g_n(y_i))x^{n-k}=0,
\]
where
\[
e_k(X_n)=e_k(x_1,\ldots,x_n)=\sum_{i_1<\cdots<i_k}x_{i_1}\cdots x_{i_k},\quad k=1,\ldots,n,
\]
are the elementary symmetric polynomials. Since $e_k(g_1(y_i),\ldots,g_n(y_i))\in K[Y_d]^G\subset K[Y_d]^GK[Z_d]^G$,
we obtain that $y_i$ is integral over $K[Y_d]^GK[Z_d]^G$. The same holds for all $z_j$. This completes the proof because the integral elements over $K[Y_d]^GK[Z_d]^G$
form a subalgebra of $K[Y_d,Z_d]$ and $Y_d\cup Z_d=\{y_i,z_j\mid i,j=1,\ldots,d\}$ generate the whole $K[Y_d,Z_d]$.
\end{proof}

\begin{corollary}\label{real analogue of Endlichkeitssatz}
Let $G$ be a nontrivial finite subgroup of $GL_d(K)$. Then
\begin{itemize}
\item[(i)] The algebra of invariants $F_d({\mathfrak B})^G$ is not finitely generated.
\item[(ii)] $F_d^2({\mathfrak B})^G$ is a finitely generated $K[Y_d]^GK[Z_d]^G$-module.
\end{itemize}
\end{corollary}

\begin{proof}
(i) As in the proof of Theorem \ref{analogue of Endlichkeitssatz} we assume that $(KX_d)^G=KX_m$, $m<d$.
Since $G(KX_m)=KX_m$, by the theorem of Maschke there exists a vector subspace $W$ of $KX_d$ such that $KX_d=KX_m\oplus W$ and $G(W)=W$.
We may assume that $W=K\{x_{m+1},\ldots,x_d\}$.
The group $G$ acts nontrivially on $W$ and hence $K[x_{m+1},\ldots,x_d]^G\not=K$. This immediately implies (i).

(ii) By the Endlichkeitssatz $K[Y_d,Z_d]^G$ is generated by a finite system of polynomials $f_1(Y_d,Z_d),\ldots,f_m(Y_d,Z_d)$. By Lemma \ref{integral polynomials}
each $f_i(Y_d,Z_d)$ is integral over $K[Y_d]^GK[Z_d]^G$. Hence there exists a positive integer $n_i$ such that all $f_i^n(Y_d,Z_d)$ can be expressed as
linear combinations of $f_i^k(Y_d,Z_d)$, $k=0,1,\ldots,n_i$, with coefficients in $K[Y_d]^GK[Z_d]^G$. Therefore $K[Y_d,Z_d]^G$ is generated as a $K[Y_d]^GK[Z_d]^G$-module
by the finite number of products
\[
f_1^{k_1}(Y_d,Z_d)\cdots f_m^{k_m}(Y_d,Z_d),\quad k_i\leq n_i, i=1,\ldots,m.
\]
Since $(R_d^2)^G\subset K[Y_d,Z_d]^G$ is a $K[Y_d]^GK[Z_d]^G$-submodule of the finitely generated $K[Y_d]^GK[Z_d]^G$-module $K[Y_d,Z_d]^G$
it is also finitely generated as a $K[Y_d]^GK[Z_d]^G$-module.
\end{proof}

\section{Analogue of the Molien formula}

The Molien formula \cite{M} gives the Hilbert series of the algebra of invariants of the finite subgroup $G$ of $GL_d(K)$:
\[
H(K[X_d]^G,t)=\frac{1}{\vert G\vert}\sum_{g\in G}\frac{1}{\det(1-gt)}.
\]
It has an analogue for free associative algebras obtained by Dicks and Formanek \cite{DiF} where the determinants are replaced by traces:
\[
H(K\langle X_d\rangle^G,t) = \frac{1}{\vert G\vert}\sum_{g\in G}\frac{1}{1-\text{tr}(g)t}.
\]
The most general case is due to Formanek \cite{F}.

\begin{theorem}\label{Molien formula of Formanek}
Let $G$ be a finite subgroup of $GL_d(K)$ and let
$\xi_1(g),\ldots,\xi_d(g)$ be the eigenvalues of $g \in G$.
If $\mathfrak V$ is a variety of algebras and $H(F_d({\mathfrak V}),t_1,\ldots,t_d)$ is the Hilbert series of
$F_d({\mathfrak V})$ considered as a ${\mathbb Z}^d$-graded vector space, then
the Hilbert series of the algebra $F_d({\mathfrak V})^G$ is
\[
H(F_d({\mathfrak V})^G,t) = \frac{1}{\vert G\vert}\sum_{g\in G}H(F_d({\mathfrak V}),\xi_1(g)t,\ldots,\xi_d(g)t).
\]
\end{theorem}

The next theorem is the analogue of the Molien formula for the algebra of invariants of a finite group acting on a free bicommutative algebra.

\begin{theorem}
For a finite subgroup $G$ of $GL_d(K)$ the Hilbert series of the algebra $F_d({\mathfrak B})^G$ is
\[
H(F_d({\mathfrak B})^G,t)=\frac{1}{\vert G\vert}\sum_{g\in G}\left(\left(\frac{1}{\det(1-gt)}-1\right)^2+\text{\rm tr}(g)t\right).
\]
\end{theorem}

\begin{proof}
Working in $R_d$, the algebra $R_d^G$ has a vector space decomposition
\[
R_d^G=(KX_d)^G\oplus (R_d^2)^G\text{ and }(R_d^2)^G\subset K[Y_d,Z_d]^G.
\]
By the properties of the Reynolds operator
\[
\rho=\frac{1}{\vert G\vert}\sum_{g\in G}g
\]
applied to a $G$-module $V$, we obtain that $V^G=\rho(V)$.
When the $G$-module $V$ is finite dimensional the theorem of Maschke gives
$V=\rho(V)\oplus\text{Ker}(\rho)=V^G\oplus\text{Ker}(\rho)$.
Since $\rho$ acts as the identity operator on $V^G$, we have that $\dim(V^G)=\text{tr}(\rho)$.
Hence
\[
\dim((KX_d)^G)=\frac{1}{\vert G\vert}\sum_{g\in G}\text{\rm tr}(g).
\]
We split $K[Y_d,Z_d]$ into a direct sum of three $G$-invariant subspaces:
\[
K[Y_d,Z_d]=K[Y_d]\oplus\omega(K[Z_d])\oplus\omega(K[Y_d])\omega(K[Z_d]).
\]
This implies the decomposition
\[
K[Y_d,Z_d]^G=K[Y_d]^G\oplus\omega(K[Z_d])^G\oplus(\omega(K[Y_d])\omega(K[Z_d]))^G
\]
which gives for the Hilbert series
\[
H(K[Y_d,Z_d]^G,t)=H(K[Y_d]^G,t)+H(\omega(K[Z_d])^G,t)+H((\omega(K[Y_d])\omega(K[Z_d]))^G,t)
\]
\[
=H(K[Y_d]^G,t)+(H(K[Z_d]^G,t)-1)+H((R_d^2)^G,t),
\]
\[
H((R_d^2)^G,t)=H(K[Y_d,Z_d]^G,t)-H(K[Y_d]^G,t)-(H(K[Z_d]^G,t)-1)
\]
\[
=\frac{1}{\vert G\vert}\sum_{g\in G}\frac{1}{\det((1-gt)^2)}-\frac{2}{\vert G\vert}\sum_{g\in G}\frac{1}{\det(1-gt)}+1
\]
\[
=\frac{1}{\vert G\vert}\sum_{g\in G}\left(\frac{1}{\det(1-gt)}-1\right)^2.
\]
This completes the proof of the theorem because
\[
H(F_d({\mathfrak B})^G,t)=H(R_d^G,t)=\dim((KX_d)^G)t+H((R_d^2)^G,t)
\]
\[
=\frac{1}{\vert G\vert}\sum_{g\in G}\left(\left(\frac{1}{\det(1-gt)}-1\right)^2+\text{\rm tr}(g)t\right).
\]
\end{proof}

\section{The Chevalley-Shephard-Todd theorem}

The Chevalley-Shephard-Todd theorem describes the algebras of invariants which have a system of algebraically independent generators.
\begin{theorem}
If $G$ is a finite subgroup of $GL_d(K)$, then $K[X_d]^G\cong K[Y_d]$ if and only if $G< GL_d(K)$ is generated by pseudo-reflections
(matrices of finite multiplicative order with all but one eigenvalues equal to $1$).
\end{theorem}

The following result shows that there is no analogue of the Chevalley-Shephard-Todd theorem for free bicommutative algebras.

\begin{theorem} Let $1\not=G\subseteq GL_d(K)$ be a finite group. Then
the algebra $F_d({\mathfrak B})^G$ is not isomorphic to a free bicommutative algebra.
\end{theorem}

\begin{proof}
By Corollary \ref{real analogue of Endlichkeitssatz} the algebra $F_d({\mathfrak B})^G$ is not finitely generated.
Hence for the proof it is sufficient to show that the algebra $F_d({\mathfrak B})$ does not contain subalgebras
isomorphic to $F_e({\mathfrak B})$ for $e>d$. For the proof, let $f_1,\ldots,f_e\in F_d({\mathfrak B})$
generate an $e$-generated free subalgebra of $F_d({\mathfrak B})$, $e>d$.
Working in $R_d$ instead of in $F_d({\mathfrak B})$ and changing linearly the generators of the subalgebra
we may assume that $f_1\in R_d^2$. Since $f_2^2$ also belongs to $R_d^2$ and $R_d^2$ is a commutative and associative algebra
we obtain that $f_1(f_2f_2)=(f_2f_2)f_1$. On the other hand,
$x_1(x_2x_2)=y_1y_2z_2\not=y_2z_1z_2=(x_2x_2)x_1$ in $R_d$.
Therefore $f_1$ and $f_2$ cannot be free generators of a free subalgebra of $F_d({\mathfrak B})$.
\end{proof}

\section{Symmetric polynomials in free bicommutative algebras}

As usually, we assume that the symmetric group $S_d$ acts on $KX_d$ as a group of permutation matrices.
We shall need the following description of the symmetric polynomials $K[Y_d,Z_d]^{S_d}$ in two sets of variables $Y_d$ and $Z_d$.
The general case of symmetric polynomials in any number of sets of variables is given by Schl\"afli \cite{Sch}, see also the books by MacMahon \cite{Mac}
and Weyl \cite{W}.

\begin{theorem}\label{symmetric polynomials in two sets of variables}
The algebra $K[Y_d,Z_d]^{S_d}$ is generated by the elementary symmetric polynomials $e_k(Y_d)$, $e_k(Z_d)$, $k=1,\ldots,d$, and their polarizations
\[
e_{p,q}(Y_d,Z_d)=\sum y_{i_1}\cdots y_{i_p}z_{j_1}\cdots z_{j_q},
\]
where $p,q\geq 1$, $p+q\leq d$, and the summation is on all tuples
$(i_1,\ldots,i_p,j_1,\ldots,j_q)$ of pairwise different integers such that
$1\leq i_1<\cdots < i_p\leq d$,  $1\leq j_1<\cdots<j_q\leq d$.
\end{theorem}

The following result describes the symmetric polynomials in $F_d({\mathfrak B})$.

\begin{theorem}\label{Symmetric bicommutative polynomails}
The vector space $F_d({\mathfrak B})^{S_d}$ decomposes as
\[
F_d({\mathfrak B})^{S_d}=Ke_1(X_d)\oplus F_d^2({\mathfrak B})^{S_d}.
\]
There exist positive integers $n_{p,q}$ such that, identifying the powers $F_d^2({\mathfrak B})$ and $R_d^2$,
the $K[Y_d]^{S_n}K[Z_d]^{S_d}$-module $(R_d^2)^{S_d}$ is generated by a finite number of products
\[
\prod_{p+q\leq d}e_{p,q}^{k_{p,q}}(Y_d,Z_d),\quad k_{p,q}\leq n_{p,q},\sum k_{p,q}\geq 1,\text{ and }e_p(Y_d)e_q(Z_d),\quad p,q=1,\ldots,d.
\]
\end{theorem}

\begin{proof}
Since $e_1(X_d)$ is the only symmetric linear polynomial in $F_d({\mathfrak B})$ it is sufficient to prove the statement for $(R_d^2)^{S_d}$.
By Theorem \ref{symmetric polynomials in two sets of variables} $(R_d^2)^{S_d}$ is spanned by the products
\[
\prod_{p=1}^de_p^{l_p}(Y_d)\prod_{p+q\leq d}e_{p,q}^{k_{p,q}}(Y_d,Z_d)\prod_{q=1}^de_q^{m_q}(Z_d)
\]
which depend both on $Y_d$ and $Z_d$. Hence $(R_d^2)^{S_d}$ is generated as a $K[Y_d]^{S_d}K[Z_d]^{S_d}$-module by the products
\[
\prod_{p+q\leq d}e_{p,q}^{k_{p,q}}(Y_d,Z_d)\text{ and }e_p(Y_d)e_q(Z_d).
\]
Now the proof follows from Corollary \ref{real analogue of Endlichkeitssatz} (ii).
\end{proof}

\begin{example}
Applying Theorem \ref{symmetric polynomials in two sets of variables} for $d=2$ we obtain that the algebra $K[Y_2,Z_2]^{S_2}$ is generated by
\[
e_1(Y_2),e_2(Y_2),e_1(Z_2),e_2(Z_2),e_{1,1}(Y_2,Z_2)=y_1z_2+y_2z_1.
\]
Direct computations show that
\[
e_{1,1}^2(Y_2,Z_2)=e_1(Y_2)e_1(Z_2)e_{1,1}(Y_2,Z_2)-e_1^2(Y_2)e_2(Z_1)-e_2(Y_2)e_1^2(Z_2)+4e_2(Y_2)e_2(Z_2).
\]
Hence as a $K[Y_2]^{S_2}K[Z_2]^{S_2}$-module $K[Y_2,Z_2]^{S_2}$ is generated by 1 and $e_{1,1}(Y_2,Z_2]$
and by Theorem \ref{Symmetric bicommutative polynomails} $(R_2^2)^{S_2}$ is generated
as a $K[Y_2]^{S_2}K[Z_2]^{S_2}$-module by $e_{1,1}(Y_2,Z_2)$ and the products $e_p(Y_2)e_q(Z_2)$, $p,q=1,2$.
\end{example}


\begin{thebibliography}{99}

\bibitem{BaCZ}
Y. Bai, Y. Chen, Z. Zhang,
{\it Gelfand-Kirillov dimension of bicommutative algebras},
Linear and Multilinear Algebra (to appear).

\bibitem{BN}
A.A. Balinskii, S.P. Novikov,
{\it Poisson brackets of hydrodynamic type, Frobenius algebras and Lie algebras} (Russian),
Dokl. Akad. Nauk SSSR {\bf 283} (1985), No. 5, 1036-1039.
Translation: Sov. Math., Dokl. {\bf 32} (1985), 228-231.

\bibitem{BCZ}
L.A. Bokut, Y. Chen, Z. Zhang,
{\it On free Gelfand-Dorfman-Novikov-Poisson algebras and a PBW theorem},
J. Algebra {\bf 500} (2018), 153-170.

\bibitem{BBDGK}
F. Benanti, S. Boumova, V. Drensky, G.K. Genov, P. Koev,
{\it Computing with rational symmetric functions and applications to invariant theory and PI-algebras},
Serdica Math. J. {\bf 38} (2012), 137-188.

\bibitem{BDD}
D. Burde, K. Dekimpe, S. Deschamps,
{\it LR-algebras}, in
C.S. Gordon et al. (Eds), New developments in Lie theory and geometry, Contemporary Mathematics {\bf 491}, 2009, 125-140.

\bibitem{Ca}
A. Cayley,
{\it On the theory of analytical forms called trees},
Phil. Mag. {\bf 13} (1857), 19-30.
Collected Math. Papers, University Press, Cambridge, Vol. 3, 1890, 242-246.

\bibitem{Ch}
C. Chevalley,
{\it Invariants of finite groups generated by reflections},
Amer. J. Math. {\bf 67} (1955), 778-782.

\bibitem{DiF}
W. Dicks, E. Formanek,
{\it Poicar\'e series and a problem of S. Montgomery},
Lin. Multilin. Algebra {\bf 12} (1982), 21-30.

\bibitem{D1}
V. Drensky,
{\it Commutative and noncommutative invariant theory},
Math. and Education in Math.,
Proc. of the 24-th Spring Conf. of the Union of Bulgar. Mathematicians,
Svishtov, April 4-7, 1995, Sofia, 1995, 14-50.

\bibitem{D2}
V. Drensky,
{\it Invariant theory of relatively free right-symmetric and Novikov algebras},
Mathematical Journal, Institute of Mathematics and Mathematical Modeling, Almaty, Kazakhstan, {\bf 16} (2016), No. 2 (60), 145-158.

\bibitem{D3}
V. Drensky,
{\it Varieties of bicommutative algebras},
Serdica Math. J. {\bf 45} (2019), 167-188.

\bibitem{DH}
V. Drensky, E. Hristova,
{\it Noncommutative invariant theory of symplectic and orthogonal groups},
Linear Algebra and its Applications {\bf 581} (2019), 198-213.

\bibitem{DZ}
V. Drensky, B.K. Zhakhayev,
{\it Noetherianity and Specht problem
for varieties of bicommutative algebras},
J. Algebra {\bf 499} (2018), 570-582.

\bibitem{DIT}
A.S. Dzhumadil'daev, N.A. Ismailov, K.M. Tulenbaev,
{\it Free bicommutative algebras},
Serdica Math. J. {\bf 37} (2011), No. 1, 25-44.

\bibitem{DL}
A.S. Dzhumadil'daev, C. L\"ofwall,
{\it Trees, free right-symmetric algebras, free Novikov algebras and identities},
The Roos Festschrift volume, {\bf 1}, Homology Homotopy Appl. {\bf 4} (2002), No. 2, Part 1, 165-190.

\bibitem{DT}
A.S. Dzhumadil'daev, K.M. Tulenbaev,
{\it Bicommutative algebras} (Russian),
Usp. Mat. Nauk {\bf 58} (2003), No. 6, 149-150.
Translation: Russ. Math. Surv. {\bf 58} (2003), No. 6, 1196-1197.

\bibitem{F}
E. Formanek,
{\it Noncommutative invariant theory},
Contemp. Math. {\bf 43}, 1985, 87-119.

\bibitem{GD}
I.M. Gelfand, I.Ya. Dorfman,
{\it Hamiltonian operators and algebraic structures structures related to them} (Russian),
Funktsional. Anal. i Prilozhen. {\bf 13} (1979), No. 4, 13-30.
Translation: Funct. Anal. Appl. {\bf 13} (1980), 248-262.

\bibitem{H}
D. Hilbert,
{\it \"Uber die Theorie der algebraischen Formen},
Math. Ann. {\bf 36} (1890), 473-534;
reprinted in ``Gesammelte Abhandlungen, Band II, Algebra,
Invariantentheorie, Geometrie'', Zweite Auflage,
Springer-Verlag, Berlin-Heidelberg- New York, 1970, 199-257.

\bibitem{Mac}
P.A. MacMahon,
{\it Combinatory Analysis}, Vol. II,
Cambridge Univ. Press, 1916.
Reprinted in one volume: New York, Chelsea, 1960.

\bibitem{M}
T. Molien,
{\it \"Uber die Invarianten der linearen
Substitutionsgruppen},
Sitz. K\"onig Preuss. Akad. Wiss. (1897), No. 52, 1152-1156.

\bibitem{No}
E. Noether,
{\it Der Endlichkeitssatz der Invarianten endlicher Gruppen},
Math. Ann. {\bf 77} (1916), 89-92;
reprinted in ``Gesammelte Abhandlungen. Collected Papers'',
Springer-Verlag, Berlin-Heidelberg-New York-Tokyo, 1983,
181-184.

\bibitem{Sch}
L. Schl\"afli,
{\it \"Uber die Resultante eines Systemes mehrerer algebraischer Gleichungen},
Vienna Academy Denkschriften {\bf 4}, 1852.

\bibitem{ST}
G.C. Shephard, J.A. Todd,
{\it Finite unitary reflection groups},
Canad. J. Math. {\bf 6} (1954), 274-304.

\bibitem{SZ}
I. Shestakov, Z. Zhang,
{\it Automorphisms of finitely generated relatively free bicommutative algebras},
J. Pure Appl. Algebra {\bf 225} (2021), No. 8, Article ID 106636, 19 p.

\bibitem{W}
H. Weyl,
{\it The Classical Groups, Their Invariants and Representations},
Princeton Univ. Press, Princeton, N.J., 1946, New Edition, 1997.

\end{thebibliography}
\end{document}